\documentclass[11pt,reqno]{amsart}
\usepackage[utf8]{inputenc}
\usepackage{amscd,amssymb,amsmath,amsthm}
\usepackage[arrow,matrix]{xy}
\usepackage{graphicx}
\usepackage{epstopdf}
\usepackage{color}
\usepackage{tikz}
\usetikzlibrary {positioning}

\newtheorem{thm}{Theorem}
\newtheorem{defn}{Definition}
\newtheorem{lemma}{Lemma}
\newtheorem{pro}{Proposition}

\makeatletter
\@namedef{subjclassname@2020}{%
  \textup{2020} Mathematics Subject Classification}
\makeatother

\numberwithin{equation}{section} 

\def\b{\beta}

\begin{document}
\title [Quadratic stochastic operators of a bisexual population]
{On dynamical systems of quadratic stochastic operators constructed for bisexual populations}

\author{Z.S. Boxonov}

\address{V.I.Romanovskiy Institute of Mathematics of Uzbek Academy of Sciences}

\begin{abstract} For two classes of bisexual populations we give a constructive description of quadratic stochastic operators which act to the  Cartesian product of standard simplexes.
We consider a bisexual population such that the set of females can be partitioned into finitely many different types indexed by $\{1,2,...,n\}$ and, similarly, that the male types are indexed by $\{1,2,...,\nu\}$. Quadratic stochastic operators were constructed for the bisexual population for the cases $n=\nu=2$ and $n=\nu=4$. In both cases, we study dynamical systems generated by the quadratic operators of the bisexual population. We find all fixed points, and limit points of the dynamical systems. Moreover, we give some biological interpretations of our results.
\end{abstract}


 \maketitle

\section{\bf{Introduction}}

The action of genes is manifested statistically in sufficiently large communities of matching individuals (belonging to the same species). These communities are called populations \cite{L}. The population exists not only in space but also in time, i.e. it has its own life cycle. The basis for this phenomenon is a reproduction by mating. Mating in a population can be free or subject to certain restrictions. The whole population in space and time comprises discrete generations $\mathbb{G}_0,$ $\mathbb{G}_1, . . . .$ The generation $\mathbb{G}_{n+1}$ is the set of individuals whose parents belong to the $\mathbb{G}_n$
generation. A state of a population is a distribution of probabilities of the different types of organisms in every generation. Type partition is called differentiation. The simplest example is sex differentiation. In the bisexual population any kind of differentiation must agree with the sex differentiation, i.e. all the organisms of one type must belong to the same sex. Thus, it is possible to speak of male and female types \cite{L}, \cite{GJ}, \cite{Rpd}.

Dynamics of such population usually investigated by quadratic stochastic operator (QSO) given on standard simplex.

Many problems of natural sciences can be solved by using nonlinear dynamical systems (see \cite{D}, \cite{Rpd}, \cite{BR}, \cite{EJX}).

The notion of QSO was first introduced in \cite{B}. Let us give basic definitions.
A QSO of a free population has the following meaning: Assume that a free population consists of $n$ elements. The set
\begin{equation}\label{Simpleks(n-1)}S^{n-1}=\left\{x=(x_{1},...,x_{n})\in\mathbb{R}^{n}: x_{i}\geq0, \sum_{i=1}^{n}x_{i}=1\right\}
\end{equation}
is called an $(n-1)$ - dimensional simplex.

A QSO that maps a simplex into itself $V:S^{n-1}\rightarrow S^{n-1}$, has the form
\begin{equation}\label{kso}
V: x'_k=\sum_{i=1}^{n} \sum_{j=1}^{n} p_{ij,k}x_{i}x_{j}, \ \ k=1,...,n,
\end{equation}
where $p_{ij,k}$ are the hereditary coefficient and
\begin{equation}\label{k=1}
p_{ij,k}\geq0,\ \ \sum_{k=1}^{n} p_{ij,k}=1,\ \ i,j,k=1,...,n.
\end{equation}
Note that each element $x\in S^{n-1}$ is a probability distribution on $E=\{1,...,n\}.$

The trajectory $\{x^{(m)}\}, m=0,1,2,...,$ for $x^{(0)}\in S^{n-1}$ under the action of the QSO (\ref{kso}) is defined as follows: $x^{(m+1)}=V(x^{(m)}), m=0,1,2,... .$

One of the main problems for operator $V$ in mathematical biology is the investigation of the asymptotic behavior of trajectories. This problem is completely solved for the Volterra QSOs (see \cite{RNG1}, \cite{RNG2}, \cite{RNGandE}) defined by equalities (\ref{kso}) and (\ref{k=1}) and the additional condition
\begin{equation}\label{vcond}
p_{ij,k}=0,\ \ if \ k \notin \{i,j\}.
\end{equation}

Let $F=\{F_1, F_2, ... , F_n\}$ be the set of female types and let $M=\{M_1, M_2, ... ,M_\nu\}$ be the set of male types. The number $n+\nu$ is called the dimension of a population. The state of a population is defined as the pair of probability distributions $x=(x_1, x_2, ... , x_n)$ and $y=(y_1, y_2, ... , y_\nu)$ on the sets $F$ and $M$ respectively(\cite{RJ}):
\begin{equation}
x_i\geq0,\ \ \sum_{i=1}^{n} x_{i}=1;\ \ y_k\geq0,\ \ \sum_{k=1}^{\nu} y_{k}=1.
\end{equation}
The space of states of this population is the Cartesian product $S^{n-1}\times S^{\nu-1}$ of the $(n-1)$-dimensional simplex $S^{n-1}$ and the $(\nu-1)$-dimensional simplex $S^{\nu-1}$. The differentiation of a population is called hereditary if, for any state $(x, y)$ of a generation $\mathbb{G}$ the state $(x', y')$ of the next generation $\mathbb{G}'$ formed by crossing and selection is uniquely determined.

The mapping $V:S^{n-1}\times S^{\nu-1}\rightarrow S^{n-1}\times S^{\nu-1}$ defined by the equality
\begin{equation}\label{W}
(x', y')=V(x, y), \ \ (x, y)\in S^{n-1}\times S^{\nu-1},
\end{equation}
is called the evolution operator. In terms of coordinates, it has the form of the system of equalities
\begin{equation}
\left\{
  \begin{array}{ll}
    x_{i}'=f_{i}(x_1, x_2, ... , x_n, y_1, y_2, ... , y_\nu), & 1\leq i\leq n, \\
    y_{k}'=g_{k}(x_1, x_2, ... , x_n, y_1, y_2, ... , y_\nu), & 1\leq k\leq \nu,
  \end{array}
\right.
\end{equation}
which are also called the evolution equalities. For any initial state $(x^{(0)}, y^{(0)})$ mapping (\ref{W}) uniquely defines the trajectory
\begin{equation}\label{orbit}
\{(x^{(t)}, y^{(t)})\}_{t=0}^{\infty}: (x^{(t+1)}, y^{(t+1)})=V((x^{(t)}, y^{(t)}))=V^{(t+1)}((x^{(0)}, y^{(0)})), \ \ t=0,1,2,....
\end{equation}
The set of the limit points of the trajectory that starts from the point $(x^{(0)}, y^{(0)})$ is called its limit set and is denoted by $\omega(x^{(0)}, y^{(0)}).$ Let us deduce evolution equations for a two-sex population. As a basis, we use the hereditary coefficients $p^{(f)}_{ik,j}$, $p^{(m)}_{ik,l}$. The quantity $p^{(f)}_{ik,j}$ is defined as the probability of the birth of an offspring of the female type $F_j$, $1\leq j\leq n$, for a mother of the type $F_i$, $1\leq i\leq n$, and a father of the type $M_k$, $1\leq k\leq \nu$. The quantity $p^{(m)}_{ik,l}$, $1\leq i\leq n$, $1\leq k, l\leq \nu$  is defined by analogy. It is obvious that
\begin{equation}\label{pfpm}
p^{(f)}_{ik,j}\geq0,\ \ \sum_{j=1}^{n} p^{(f)}_{ik,j}=1;\ \ p^{(m)}_{ik,l}\geq0,\ \ \sum_{l=1}^{\nu} p^{(m)}_{ik,l}=1.
\end{equation}
The hereditary coefficients take into account a number of factors such as, e.g., the recombination process, selection of gametes, mutations, and differential birth rate.
Let $(x, y)$ be a state in the generation $\mathbb{G}$. In the next generation $\mathbb{G}'$, at the time of its germination, the probabilities of the types are determined according to the formula of total probability:
\begin{equation}\label{ksobp}
V:\left\{
  \begin{array}{ll}
    x_{j}'=\sum_{i=1}^{n} \sum_{k=1}^{\nu} p^{(f)}_{ik,j}x_{i}y_{k}, & 1\leq j\leq n, \\[2mm]
    y_{l}'=\sum_{i=1}^{n} \sum_{k=1}^{\nu} p^{(m)}_{ik,l}x_{i}y_{k}, & 1\leq l\leq \nu.
  \end{array}
\right.
\end{equation}
The operator is called quadratic stochastic operator of a bisexual population (QSOBP).

Some QSOs of the bisexual population were studied in \cite{RJ}, \cite{LR}, \cite{DOR}, \cite{CJR}, \cite{GJ}, \cite{Rpd}.

\textbf{The main problem:} Given an operator $V$ and initial point $z^{(0)}\in S^{n-1}\times S^{\nu-1}$ what ultimately happens with the trajectory $z^{(n)}=V(z^{(n-1)}), n=1,2,...?$ Does the limit $\lim_{n\rightarrow\infty}z^{(n)}$ exist? If not what is the set of limit points of the sequence? Is this set finite or infinite \cite{D}.

This problem is a rather difficult in general, because the dynamical system generated by operator (\ref{ksobp}) is a complicated nonlinear system. Therefore, we shall consider two particular cases of the system and for them we completely study the main problem.

\section{\bf{Construction of QSOBP for finite $E$}}

Let $(E, \mathcal{F})$ be a measurable space and $S(E, \mathcal{F})$ be the set of all probability measures on $(E, \mathcal{F})$.
When $E$ is finite, QSO on $S(E, \mathcal{F})=S^{n-1}$ is defined as in (\ref{kso}).

Let $G=(\Lambda, L)$ be finite graph witout loops and multiple edges, where $\Lambda$ is the set of vertexes and $L$ is the set of edges of the graph.

Furthermore, let $\Phi$ be a finite set, called the set of alleles (in problems of statistical mechanics, $\Phi$ is called the range of spin). The function $\sigma:\Lambda\rightarrow\Phi$ is called a cell (in mechanics it is called configuration). Denote by $\Omega$ the set of all cells, this set corresponds to $E$ in (\ref{kso}). Let $S(\Lambda, \Phi)$ be the set of all probability measures defined on the finite set $\Omega.$

Let $\{\Lambda_i, i=1,2,...,n\}$ be the set of maximal connected subgraphs (components) of the graph $G.$

Let us fix a set $F\subset\Omega$ and call it the set of "female", while the set $M=\Omega\setminus F$ is called the set of "male". For a configuration $\sigma\in\Omega$ denote by $\sigma(A)$ its "projection" (or "restruction") to $A\subset\Lambda: \sigma(A)=\{\sigma(x)\}_{x\in A}.$ Fix two cells $\sigma^f_1 \in F, \sigma^m_2 \in M,$ and put
$$\Omega^f(G,\sigma^f_1,\sigma^m_2)=\{\sigma\in F: \sigma(\Lambda_i)=\sigma^f_1(\Lambda_i) \ or \ \sigma(\Lambda_i)=\sigma^m_2(\Lambda_i) \ for \ all \ i=1,...,n\},$$
$$\Omega^m(G,\sigma^f_1,\sigma^m_2)=\{\sigma\in M: \sigma(\Lambda_i)=\sigma_1(\Lambda_i) \ or \ \sigma(\Lambda_i)=\sigma_2(\Lambda_i) \ for \ all \ i=1,...,n\}.$$
Now let $\mu^f\in S(\Lambda, \Phi)$ $(resp.\ \mu^m\in S(\Lambda, \Phi))$ be a probability measure defined on $F$ $(resp.\ M)$ such that $\mu^f(\sigma^f)>0, \mu^m(\sigma^m)>0$ for any cell $\sigma^f\in F$, $\sigma^m\in M$; i. e. $\mu^f$, $\mu^m$,  are Gibbs measures with some potential \cite{P}, \cite{GR}. The heredity coefficients $p^{(f)}_{\sigma^f_1\sigma^m_2,\sigma^f}$, $p^{(m)}_{\sigma^f_1\sigma^m_2,\sigma^m}$ are defined as
\begin{equation}\label{hercoef}
\begin{array}{ll}
    p^f_{\sigma^f_1\sigma^m_2,\sigma^f}=\left\{
                                            \begin{array}{ll}
                                              \frac{\mu^f(\sigma)}{\mu^f(\Omega^f(G,\sigma^f_1,\sigma^m_2))}, & \hbox{if \ \ $\sigma\in \Omega^f(G,\sigma^f_1,\sigma^m_2)$} \\
                                              0, & \hbox{otherwise,}
                                            \end{array}
                                          \right.\\[5mm]
    p^m_{\sigma^f_1\sigma^m_2,\sigma^m}=\left\{
                                            \begin{array}{ll}
                                              \frac{\mu^m(\sigma)}{\mu^m(\Omega^m(G,\sigma^f_1,\sigma^m_2))}, & \hbox{if \ \ $\sigma\in \Omega^m(G,\sigma^f_1,\sigma^m_2)$} \\
                                              0, & \hbox{otherwise.}
                                            \end{array}
                                          \right.
  \end{array}
\end{equation}
Obviously, $p^f_{\sigma^f_1\sigma^m_2,\sigma^f}\geq0$, $p^m_{\sigma^f_1\sigma^m_2,\sigma^m}\geq0$ and $\sum\limits_{\sigma\in F}p^f_{\sigma^f_1\sigma^m_2,\sigma^f}=1$, $\sum\limits_{\sigma\in M}p^{(m)}_{\sigma^f_1\sigma^m_2,\sigma^m}=1.$

The QSOBP $V\equiv V_\mu$ acting on the simplex $S(\Lambda, \Phi)$ and determined by coefficients (\ref{hercoef}) is defined as follows: for an arbitrary measure $\lambda\in S(\Lambda, \Phi)$, the measure $V(\lambda)=\lambda'\in S(\Lambda, \Phi)$ is defined by the equality
\begin{equation}\label{m.kso}
\left\{\begin{array}{ll}
    \lambda'(\sigma^f)=\sum\limits_{\sigma^f_1\in F, \sigma^m_2\in M}p^f_{\sigma^f_1\sigma^m_2,\sigma^f}\lambda(\sigma^f_1)\lambda(\sigma^m_2), \ \ \sigma^f\in F,\\[5mm]
    \lambda'(\sigma^m)=\sum\limits_{\sigma^f_1\in F, \sigma^m_2\in M}p^m_{\sigma^f_1\sigma^m_2,\sigma^m}\lambda(\sigma^f_1)\lambda(\sigma^m_2), \ \ \sigma^m\in M.
  \end{array}\right.
\end{equation}
Thus Gibbs measures arise in a natural way when QSOBPs are constructed.
\begin{thm}\label{t1} The QSOBP (\ref{m.kso}) is the identity operator if and only if graph G is connected.
\end{thm}
\begin{proof} Let G be a connected graph. Then, fix two cells $\sigma^f_1 \in F, \sigma^m_2 \in M,$ we have
$$\Omega^f(G,\sigma^f_1,\sigma^m_2)=\{\sigma^f_1\},\ \ \Omega^m(G,\sigma^f_1,\sigma^m_2)=\{\sigma^m_2\}$$
and  the heredity coefficients as follows:
\begin{equation}\label{herid}
\begin{array}{ll}
    p^f_{\sigma^f_1\sigma^m_2,\sigma^f}=\left\{
                                            \begin{array}{ll}
                                              1, & \hbox{if \ \ $\sigma\in \Omega^f(G,\sigma^f_1,\sigma^m_2)$} \\
                                              0, & \hbox{otherwise,}
                                            \end{array}
                                          \right.\\[5mm]
    p^m_{\sigma^f_1\sigma^m_2,\sigma^m}=\left\{
                                            \begin{array}{ll}
                                              1, & \hbox{if \ \ $\sigma\in \Omega^m(G,\sigma^f_1,\sigma^m_2)$} \\
                                              0, & \hbox{otherwise.}
                                            \end{array}
                                          \right.
  \end{array}
\end{equation}
If the coefficients of QSOBP (\ref{m.kso}) will be described such (\ref{herid}), we have
\begin{equation}\label{i.kso}
\left\{\begin{array}{ll}
    \lambda'(\sigma^f)=\lambda(\sigma^f)\sum\limits_{\sigma^m_2\in M}\lambda(\sigma^m_2)=\lambda(\sigma^f), \ \  \sigma^f\in F,\\[5mm]
    \lambda'(\sigma^m)=\lambda(\sigma^m)\sum\limits_{\sigma^f_1\in F}\lambda(\sigma^f_1)=\lambda(\sigma^m), \ \ \sigma^m\in M.
  \end{array}\right.
\end{equation}
\end{proof}

\section{\bf{Construction of a QSOBP on $S^1\times S^1$}}

Let $G$ be a graph consisting of only two vertices. Moreover, let $\Phi=\{1, 2\}.$

$\Omega=\{\sigma_1, \sigma_2, \sigma_3, \sigma_4\}$ is the set of all cells, where $\sigma_1=\{1,1\}$, $\sigma_2=\{1,2\}$, $\sigma_3=\{2,1\}$, $\sigma_4=\{2,2\}$.

$F=\{\sigma^f_i\}$,\ $i=1,2$ is the set of females, $M=\{\sigma^f_j\}$,\ $j=3,4$ is the set of males.
\begin{equation}\label{exam}
\begin{array}{ll}
    \Omega^f(G,\sigma^f_i,\sigma^m_j)=\left\{
                                            \begin{array}{ll}
                                              \{\sigma^f_1, \sigma^f_2\}, & \hbox{if \ \ $i=2$, $j=3,$} \\
                                              \{\sigma^f_i\}, & \hbox{otherwise,}
                                            \end{array}
                                          \right.\\[5mm]
    \Omega^m(G,\sigma^f_i,\sigma^m_j)=\left\{
                                            \begin{array}{ll}
                                              \{\sigma^m_3, \sigma^m_4\}, & \hbox{if \ \ $i=2$, $j=3,$} \\
                                              \{\sigma^m_j\}, & \hbox{otherwise.}
                                            \end{array}
                                          \right.
  \end{array}
\end{equation}
From (\ref{exam}) we obtain hereditary coefficients as follows:
\begin{equation}\label{ehc}
\begin{array}{ll}
    p^f_{\sigma^f_i\sigma^m_j,\sigma^f_k}=\left\{
                                            \begin{array}{ll}
                                              a, & \hbox{if \ \ $i=2,\ j=3, \ k=1,$} \\
                                             1-a, & \hbox{if \ \ $i=2,\ j=3, \ k=2,$} \\
                                             1, & \hbox{if \ \ $i=k=1,\ j=3,4$ \ or\  $i=k=2,\ j=4,$} \\
                                             0, & \hbox{otherwise,}
                                            \end{array}
                                          \right.\\[5mm]
    p^m_{\sigma^f_i\sigma^m_j,\sigma^m_l}=\left\{
                                            \begin{array}{ll}
                                              b, & \hbox{if \ \ $i=2,\ j=3, \ l=3,$} \\
                                             1-b, & \hbox{if \ \ $i=2,\ j=3, \ l=4,$} \\
                                             1, & \hbox{if \ \ $i=1, \ j=l=3,4$ \ or\  $i=2,\ j=l=4,$} \\
                                             0, & \hbox{otherwise.}
                                            \end{array}
                                          \right.
  \end{array}
\end{equation}
where $$a=\frac{\mu^f(\sigma^f_1)}{\mu^f(\sigma^f_1)+\mu^f(\sigma^f_2)},\ b=\frac{\mu^m(\sigma^m_3)}{\mu^m(\sigma^m_3)+\mu^m(\sigma^m_4)}.$$

If the coefficients of QSOBP (\ref{m.kso}) given as (\ref{ehc}), then we have
\begin{equation}\label{exam.kso}
V:\left\{\begin{array}{ll}
    \lambda'(\sigma^f_1)=\lambda(\sigma^f_1)\lambda(\sigma^m_3)+\lambda(\sigma^f_1)\lambda(\sigma^m_4)+a\lambda(\sigma^f_2)\lambda(\sigma^m_3),\\
    \lambda'(\sigma^f_2)=(1-a)\lambda(\sigma^f_2)\lambda(\sigma^m_3)+\lambda(\sigma^f_2)\lambda(\sigma^m_4),\\
    \lambda'(\sigma^m_3)=\lambda(\sigma^f_1)\lambda(\sigma^m_3)+b\lambda(\sigma^f_2)\lambda(\sigma^m_3),\\
    \lambda'(\sigma^m_4)=\lambda(\sigma^f_1)\lambda(\sigma^m_4)+\lambda(\sigma^f_2)\lambda(\sigma^m_4)+(1-b)\lambda(\sigma^f_2)\lambda(\sigma^m_3).
  \end{array}\right.
\end{equation}
Denote
\begin{equation}\label{xyuv}
x_1=\lambda(\sigma^f_1),\ x_2=\lambda(\sigma^f_2), \ y_1=\lambda(\sigma^m_3), \ y_2=\lambda(\sigma^m_4).
\end{equation}
From $\lambda(\sigma^f_1)+\lambda(\sigma^f_2)=1$, $\lambda(\sigma^m_3)+\lambda(\sigma^m_4)=1$ and (\ref{xyuv}), we can write operator $V$ in the form:
\begin{equation}\label{xyuvkso}
V:\left\{\begin{array}{ll}
    x_1'=x_1+ax_2y_1,\\
    x_2'=(1-a)x_2y_1+x_2y_2,\\
    y_1'=x_1y_1+bx_2y_1,\\
    y_2'=y_2+(1-b)x_2y_1.
  \end{array}\right.
\end{equation}
Since $(x_1,x_2;y_1,y_2)\in S^1\times S^1$, we have that $x_2=1-x_1$ and $y_2=1-y_1$ (we denote $x_1=x, y_1=y$).
One can rewrite the QSOBP (\ref{xyuvkso}) in the following form
\begin{equation}\label{wkso}
W:\left\{\begin{array}{ll}
    x'=x+a(1-x)y,\\
    y'=y(x+b(1-x)),
    \end{array}\right.
\end{equation}
where $a,b\in(0,1),$ and $W:S^1_0\rightarrow S^{1}_{0},$ $S^1_0=\{(x,y)\in\mathbb{R}^2: x,y\in[0,1]\}.$ Thus we get a quadratic operator with two independent parameters.
\subsection{Fixed points of (\ref{xyuvkso})}\

Let us first find the fixed points of the operator $W$, defined as (\ref{wkso}).
To find fixed points of (\ref{wkso}) we should solve the following
\begin{equation}
\begin{array}{ll}
    x=x+a(1-x)y,\\
    y=y(x+b(1-x)).
    \end{array}
\end{equation}
Denote by $Fix(W)$ the set of all fixed points of the operator $W$ given by (\ref{wkso}).

The following proposition is straightforward.
\begin{pro}\label{wfix} The set $Fix(W)$ has the following form
$$Fix(W)=Z_1\bigcup Z_2$$
where
\begin{equation}\label{z1z2z3}
Z_1=\{(x,y): x\in [0,1), y=0\},\  Z_2=\{(x,y): x=1, y\in[0,1]\}.
\end{equation}
\end{pro}

From Proposition \ref{wfix} we obtain the following on fixed points of the operator (\ref{xyuvkso}).
\begin{pro}\label{vfix} The set $Fix(V)$ has the following form
$$Fix(V)=\mathcal{Z}_{1}\bigcup \mathcal{Z}_{2}$$
where
\begin{equation}\label{v1v2} \begin{array}{cc}
\mathcal{Z}_{1}=\{(x_1,x_2;y_1,y_2): x_1\in [0,1),x_2=1-x_1, y_1=0,y_2=1\},\\
\mathcal{Z}_{2}=\{(x_1,x_2;y_1,y_2): x_1=1,x_2=0;y_1\in[0,1],y_2=1-y_1\}.\end{array}
\end{equation}
\end{pro}
\subsection{Type of fixed points}\

Now we shall examine the type of the fixed points.
\begin{defn}\label{d1}
(see\cite{D}) A fixed point $s$ of the operator $W$ is called hyperbolic if its Jacobian $J$ at $s$ has no eigenvalues on the unit circle.
\end{defn}

\begin{defn}\label{d2}
(see\cite{D}) A hyperbolic fixed point $s$ called:

1) attracting if all the eigenvalues of the Jacobi matrix $J(s)$ are less than 1 in absolute value;

2) repelling if all the eigenvalues of the Jacobi matrix $J(s)$ are greater than 1 in absolute value;

3) a saddle otherwise.
\end{defn}
To find the type of a fixed point of the operator (\ref{wkso}) we write the Jacobi matrix:
$$J_{W}=\left(%
\begin{array}{cc}
  1-ay & a(1-x) \\
  y(1-b) & x(1-b)+b \\
\end{array}%
\right).$$

It is clear that for any values of the parameters $a,b\in(0,1)$, one of eigenvalues of a Jacobian matrix at fixed points $\forall z_1\in Z_1$ and $\forall z_2\in Z_2\setminus\{(1,0)\} $    is always 1 and the other one is less than 1 while both eigenvalues of the Jacobian matrix at a fixed point $z_3=(1,0)$  are 1.

For the type of fixed points of the operator $W$ defined by (\ref{wkso}), the following Proposition holds:

\begin{pro}
The fixed points $z_1\in Z_1$, $z_2\in Z_2$ are non-hyperbolic (but $z_1\in Z_1$, $z_2\in Z_2\setminus\{z_3\}$ are semi-attracting \footnote{meaning that the second eigenvalue is less than 1 in absolute value.}).
\end{pro}
\subsection{The set of limit points}\

In this subsection for any initial point $(x^{(0)}, y^{(0)})\in S_0^1$ we investigate behavior of the trajectories $(x^{(n)},y^{(n)})=W^n(x^{(0)}, y^{(0)}), n\geq1.$

A set $A$ is called invariant with respect to $W$ if $W(A)\subset A.$

Denote
\begin{equation}\label{defMc}
M_c=\{(x,y)\in S^1_0:\frac{x}{a}+\frac{y}{1-b}=c\},\  c=const.
\end{equation}
\begin{lemma}\label{lem1} The sets $M_c$ are invariant with respect to the operator $W$, i.e. $W(M_c)\subset M_c$.
\end{lemma}
\begin{proof}  From the first and second equalities of (\ref{wkso}) we find
$$\frac{x'-x}{a}=(1-x)y,\ \frac{y-y'}{1-b}=(1-x)y.$$
Consequently
$$\frac{x'}{a}+\frac{y'}{1-b}=\frac{x}{a}+\frac{y}{1-b}.$$
\end{proof}
Note that
\begin{equation}\label{Mc}
S^1_0=\bigcup_cM_c.
\end{equation}
By (\ref{Mc}) it suffices to study the trajectories on each invariant set.
\begin{thm}\label{Wlimit} If $(x^{(0)}, y^{(0)})\in M_c\setminus Fix(W)$ then for the operator (\ref{wkso}) the following holds
$$\lim_{n\to \infty}W^n((x^{(0)}, y^{(0)}))=\left\{\begin{array}{ll}
 z^*_1=(ac, 0), &  \mbox{if} \ \ ac<1,\\[2mm]
 z^*_2=\left(1,\frac{(ac-1)(1-b)}{a}\right), &  \mbox{if} \ \ ac\geq1,\\
 \end{array}\right.$$
where $z^*_i\in Z_i, \ i=1, 2$; $Z_i$ are defined in (\ref{z1z2z3}) and $c=x^{(0)}/a+y^{(0)}/(1-b).$
\end{thm}
\begin{proof}
From (\ref{wkso}) we get
\begin{equation}\label{wn+1}
W^{(n+1)}((x^{(0)}, y^{(0)})):\left\{\begin{array}{ll}
    x^{(n+1)}-x^{(n)}=a(1-x^{(n)})y^{(n)}\geq0,\\
    y^{(n+1)}-y^{(n)}=(1-x^{(n)})y^{(n)}(b-1)\leq0.
    \end{array}\right.
\end{equation}
Thus $x^{(n)}$ is a non-decreasing sequence, which bounded from above by 1 and $y^{(n)}$ sequence is non-increasing and with lower bound 0. Consequently, $x^{(n)}$ and $y^{(n)}$ has a limit say $x^*(\leq1)$, $y^*(\geq0)$, respectively.

From Lemma \ref{lem1}, we have
\begin{equation}\label{Mc+1}
\frac{x^{(n)}}{a}+\frac{y^{(n)}}{1-b}=\frac{x^{(0)}}{a}+\frac{y^{(0)}}{1-b}=c.
\end{equation}
By (\ref{wn+1}) and (\ref{Mc+1}) we have
\begin{equation}\label{x*y*}
\left\{\begin{array}{ll}
    a(1-x^*)y^*=0,\\
    (b-1)(1-x^*)y^*=0.\\
    \frac{x^*}{a}+\frac{y^*}{1-b}=c
    \end{array}\right.
\end{equation}

The limit of the sequence $x^{(n)}$ is equal to either $x^*<1$ or $x^*=1$. First, consider the case  $x^*<1$.   The first and second equations in (\ref{x*y*}) derive $y^*=0$. By substituting this value to the third equation in (\ref{x*y*}) we obtain $x^*=ac<1$. Now consider the case  $x^*=1$.  One substitutes $x^*=1$ into the third equation in (\ref{x*y*}) and gets $y^*=\frac{(ac-1)(1-b)}{a}$.
\end{proof}

The following theorem  describes the trajectory of any point $z$ in $S^1\times S^1$.
\begin{thm}\label{v1set} For the operator $V$ given by (\ref{xyuvkso}) (i.e. under condition $a,b\in(0,1)$) and for any initial point $z=\left(x_1^{(0)},x_2^{(0)}; y_1^{(0)},y_2^{(0)}\right)\in S^1\times S^1\setminus Fix(V)$ the following hold
$$\lim_{n\to \infty}V^{(n)}\left(z\right)=\left\{\begin{array}{ll}
z^*_1=\left(ac,1-ac;0,1\right), &  \mbox{if} \ \ ac<1,\\[2mm]
z^*_2=\left(1,0;\frac{(ac-1)(1-b)}{a},1-\frac{(ac-1)(1-b)}{a}\right), &  \mbox{if} \ \ ac\geq1,\\
\end{array}\right.$$
where $z^*_i\in \mathcal{Z}_i, \ i=1, 2$; $\mathcal{Z}_i$ are defined in (\ref{v1v2}) and $c=x_1^{(0)}/a+y_1^{(0)}/(1-b).$
\end{thm}
\begin{proof}
Easily deduced from Theorem \ref{Wlimit}.
\end{proof}

\section{\bf{An QSOBP on $S^3\times S^3$}}
Let graph $G$ consists one edge and three vertices. Moreover, let $\Phi=\{1, 2\}.$\\
$\Omega=\{\sigma_1, \sigma_2, \sigma_3, \sigma_4, \sigma_5, \sigma_6, \sigma_7, \sigma_8\}$ is the set of all cells, where
\begin{center}$\begin{array}{ll}\sigma_1=\{1-1,1\}, \sigma_2=\{1-2,1\}, \sigma_3=\{2-1,2\}, \sigma_4=\{2-2,2\},\\[2mm]
\sigma_5=\{1-1,2\}, \sigma_6=\{1-2,2\}, \sigma_7=\{2-1,1\}, \sigma_8=\{2-2,1\}. \end{array}$\end{center}
$F=\{\sigma^f_i\}$,\ $i=1,2,3,4$ is the set of females, $M=\{\sigma^f_j\}$,\ $j=5,6,7,8$ is the set of males.
\begin{equation}\label{exam2}
\begin{array}{ll}
    \Omega^f(G,\sigma^f_i,\sigma^m_j)=\left\{
                                            \begin{array}{ll}
                                              \{\sigma^f_1, \sigma^f_2\}, & \hbox{if \ \ $i=1$, $j=6,$\ or \ $i=2$, $j=5,$} \\
                                              \{\sigma^f_3, \sigma^f_4\}, & \hbox{if \ \ $i=3$, $j=8,$\ or \ $i=4$, $j=7,$} \\
                                              \{\sigma^f_i\}, & \hbox{otherwise,}
                                            \end{array}
                                          \right.\\[5mm]
    \Omega^m(G,\sigma^f_i,\sigma^m_j)=\left\{
                                            \begin{array}{ll}
                                              \{\sigma^m_5, \sigma^m_6\}, & \hbox{if \ \ $i=1$, $j=6,$\ or \ $i=2$, $j=5,$} \\
                                              \{\sigma^m_7, \sigma^m_8\}, & \hbox{if \ \ $i=3$, $j=8,$\ or \ $i=4$, $j=7,$} \\
                                              \{\sigma^m_j\}, & \hbox{otherwise.}

                                            \end{array}
                                          \right.
  \end{array}
\end{equation}
From (\ref{exam2}) we obtain hereditary coefficients as follows:
\begin{equation}\label{ehc2}
\begin{array}{ll}
    p^f_{\sigma^f_i\sigma^m_j,\sigma^f_k}=\left\{
                                            \begin{array}{ll}
                                              a, & \hbox{if \ \ $i=k=1, j=6,\ or\ i=2, j=5, k=1,$} \\
                                              b, & \hbox{if \ \ $i=k=3, j=8,\ or\ i=4, j=7, k=3,$} \\
                                             1-a, & \hbox{if \ \ $i=k=2, j=5,\ or\ i=1, j=6, k=2,$} \\
                                             1-b, & \hbox{if \ \ $i=k=4, j=7,\ or\ i=3, j=8, k=4,$} \\
                                             1, & \hbox{if \ \ $i=k=1, j=5,7,8,\ or\ i=k=2, j=6,7,8,$}\\
                                             &  \hbox{or\ \ $i=k=3, j=5,6,7,\ or \ i=k=4, j=5,6,8$} \\
                                             0, & \hbox{otherwise,}
                                            \end{array}
                                          \right.\\[5mm]
    p^m_{\sigma^f_i\sigma^m_j,\sigma^m_l}=\left\{
                                            \begin{array}{ll}
                                              c, & \hbox{if \ \ $j=l=5, i=2,\ or\ i=1, j=6, l=5,$} \\
                                              d, & \hbox{if \ \ $j=l=7, i=4,\ or\ i=3, j=8, l=7,$} \\
                                             1-c, & \hbox{if \ \ $j=l=6, i=1,\ or\ i=2, j=5, l=6,$} \\
                                             1-d, & \hbox{if \ \ $j=l=8, i=3,\ or\ i=4, j=7, l=8,$} \\
                                             1, & \hbox{if \ \ $j=l=5, i=1,3,4,\ or\ j=l=6, i=2,3,4,$}\\
                                             &  \hbox{or\ \ $j=l=7, i=1,2,3,\ or \ j=l=8, i=1,2,4$} \\
                                             0, & \hbox{otherwise,}
                                            \end{array}
                                          \right.
  \end{array}
\end{equation}
where $$a=\frac{\mu^f(\sigma^f_1)}{\mu^f(\sigma^f_1)+\mu^f(\sigma^f_2)},\ b=\frac{\mu^f(\sigma^f_3)}{\mu^f(\sigma^f_3)+\mu^f(\sigma^f_4)},\ c=\frac{\mu^m(\sigma^m_5)}{\mu^m(\sigma^m_5)+\mu^m(\sigma^m_6)},\ d=\frac{\mu^m(\sigma^m_7)}{\mu^m(\sigma^m_7)+\mu^m(\sigma^m_8)}.$$

If the coefficients of QSOBP (\ref{m.kso}) defined as (\ref{ehc2}), then we have $V:S^3\times S^3\rightarrow S^3\times S^3,$
\begin{equation}\label{exam2.kso}
V:\left\{\begin{array}{ll}
x'_1=x_1-(1-a)x_1y_2+ax_2y_1\\[2mm]
x'_2=x_2-ax_2y_1+(1-a)x_1y_2\\[2mm]
x'_3=x_3-(1-b)x_3y_4+bx_4y_3\\[2mm]
x'_4=x_4-bx_4y_3+(1-b)x_3y_4\\[2mm]
y'_1=y_1-(1-c)x_2y_1+cx_1y_2\\[2mm]
y'_2=y_2-cx_1y_2+(1-c)x_2y_1\\[2mm]
y'_3=y_3-(1-d)x_4y_3+dx_3y_4\\[2mm]
y'_4=y_4-dx_3y_4+(1-d)x_4y_3\\
 \end{array}\right.
\end{equation}
where $$x'_1=\lambda'(\sigma^f_1),\ x'_2=\lambda'(\sigma^f_2),\ x'_3=\lambda'(\sigma^f_3),\ x'_4=\lambda'(\sigma^f_4),$$
$$y'_1=\lambda'(\sigma^m_5),\ y'_2=\lambda'(\sigma^m_6),\ y'_3=\lambda'(\sigma^m_7),\ y'_4=\lambda'(\sigma^m_8).$$
For $a_0, c_0\in(0,1)$ we denote
\begin{equation}\label{a0c0}
I_{a_0c_0}=\{z\in S^3\times S^3: x_1+x_2=a_0, x_3+x_4=1-a_0, y_1+y_2=c_0, y_3+y_4=1-c_0\}.
\end{equation}
The following lemma is useful
\begin{lemma}\label{lemmaI}
For any fixed $a_0, c_0\in(0,1)$ we have
\begin{itemize}
  \item [(1)] for any $z\in S^3\times S^3$ the following holds
$$z'=V(z)\in I_{a_0c_0},\ i.e.,\  V(S^3\times S^3)\subset I_{a_0c_0};$$
  \item [(2)] the set $I_{a_0c_0}$ is invariant with respect to $V$, i.e., $V(I_{a_0c_0})\subset I_{a_0c_0}$.
\end{itemize}
\end{lemma}
\begin{proof} The proof follows from the following easily checked equality
$$x'_1+x'_2=x_1+x_2,\ x'_3+x'_4=x_3+x_4,\ y'_1+y'_2=y_1+y_2,\ y'_3+y'_4=y_3+y_4.$$
\end{proof}
For each fixed $a_0, c_0\in(0,1)$, by this Lemma \ref{lemmaI}, the investigation of the sequence $z^{(t)}=V(z^{(t-1)}),\  t=1,2,...,$ for each point $z^{(0)}\in S^3\times S^3$ is reduced to the case $z^{(0)}\in I_{a_0c_0}$. Therefore we are interested to the following dynamical system:
$$z^{(0)}\in I_{a_0c_0},\ z^{(1)}=V(z^{(0)}),\ z^{(2)}=V(z^{(1)}),...$$
the main problem is to study the limit
$$\lim\limits_{m\rightarrow\infty}z^{(m)}=\lim\limits_{m\rightarrow\infty}V^m(z^{(0)}).$$
The restriction on $I_{a_0c_0}$ of the operator $V$, denoted simply by $W_0$, has the form (we denote $x_1=x,\ y_1=y,\ x_3=u,\ y_3=v$)
\begin{equation}\label{w0kso}
W_0:\left\{\begin{array}{ll}
    x'=x-(1-a)x(c_0-y)+a(a_0-x)y,\\
    u'=u-(1-b)u(1-c_0-v)+b(1-a_0-u)v,\\
    y'=y-(1-c)(a_0-x)y+cx(c_0-y),\\
    v'=v-(1-d)(1-a_0-u)v+du(1-c_0-v),
    \end{array}\right.
\end{equation}
where
\begin{equation}\label{parshart}
a,b,c,d\in(0,1),\ a_0,c_0\in(0,1).
\end{equation}
Thus we get a quadratic operator with six independent parameters.

It can be seen from (\ref{w0kso}) that the variables in the first and third equations are not depended on the ones in the second and fourth equations. Thus, the study of the dynamics of the operator $W_0$ leads to the study of the dynamics of the operators $W_{1}$ (the first and third equations) and $W_{2}$ (the second and fourth equations).

If we suppose
\begin{equation}\label{x-u}
x\leftrightarrow v,\ y\leftrightarrow v, a\leftrightarrow b,\ c\leftrightarrow d, a_0\leftrightarrow 1-a_0,\ c\leftrightarrow 1-c_0
\end{equation}
then the operators $W_{1}$ and $W_{2}$ are exactly the same. Therefore, by studying the dynamics of operator $W_{1}$, it is possible to analyze  the dynamics of operators $W_2$, $W_0$ and $V$.

Consider the operator $W_{1}:S_{1}\rightarrow S_{1},$ $S_{1}=\{(x,y)\in\mathbb{R}^2: x\in(0,a_0), y\in(0,c_0)\},$
\begin{equation}\label{wxykso}
W_{1}:\begin{array}{ll}
    x'=x-(1-a)x(c_0-y)+a(a_0-x)y,\\
    y'=y-(1-c)(a_0-x)y+cx(c_0-y).
    \end{array}
\end{equation}
\subsection{Fixed points of (\ref{wxykso})}\

Let us first find the fixed points of the operator $W_{1}.$
To find fixed points of $W_{1}$ we should solve the following
\begin{equation}\label{wxyfix}
\begin{array}{ll}
    (1-a)x(c_0-y)=a(a_0-x)y,\\
    (1-c)(a_0-x)y=cx(c_0-y).
    \end{array}
\end{equation}
By (\ref{wxyfix}), we have
$$(a_0-x)y=\frac{1-a}{a}x(c_0-y)=\frac{c}{1-c}x(c_0-y).$$

If $a+c\neq1$, then $x(c_0-y)=0$ i.e. $x=0$ or $y=c_0.$ In this case, the system (\ref{wxyfix} )  have two solutions $(x,y)=(0,0)$ and $(x,y)=(a_0,c_0)$.
If $a+c=1$, then  the system (\ref{wxyfix})  has infinitely many solutions of the form $(x, \frac{cc_0x}{aa_0+(c-a)x})$.

To sum up,  we get the following  (use (\ref{x-u}) to find the fixed points of $W_{2}$)
\begin{pro}\label{xyuvfix} The sets $Fix(W_{1}), Fix(W_{2})$ has the following form
\begin{equation}\label{xyfix}
\begin{array}{ll}
    Fix(W_{1})=\left\{
                                            \begin{array}{ll}
                                              (0,0)\ or\ (a_0,c_0), & \hbox{if \ $a+c\neq1$}\\
                                              (x, y), & \hbox{if \ $a+c=1$}\end{array}\right. \end{array}
\end{equation}
\begin{equation}\label{uvfix}
\begin{array}{ll}
    Fix(W_{2})=\left\{
                                            \begin{array}{ll}
                                              (0,0)\ or\ (1-a_0,1-c_0), & \hbox{if \ $b+d\neq1$}\\
                                              (u, v), & \hbox{if \ $b+d=1$}\end{array}\right. \end{array}
\end{equation}
where   $$y=\frac{cc_0x}{aa_0+(c-a)x}, v=\frac{d(1-c_0)u}{b(1-a_0)+(d-b)u.}$$
\end{pro}
From Proposition \ref{xyuvfix} and (\ref{a0c0}) we obtain the following on fixed points of the operator (\ref{exam2.kso}).
\begin{pro}\label{pro5} The fixed points of the operator $V$ defined by equality (\ref{exam2.kso}) are as follows:
\begin{equation}
\begin{array}{ll}
    Fix(V)=\left\{
                                            \begin{array}{ll}
                                              (0,a_0,0,1-a_0;0,c_0,0,1-c_0) \ or \\ (a_0,0,1-a_0,0;c_0,0,1-c_0,0), & \hbox{if \ $a+c\neq1$, $b+d\neq1$}\\
                                              (a_0,0,1-a_0,0;c_0,0,1-c_0,0) \ or \\ (x_1,a_0-x_1,0,1-a_0;y_1,c_0-y_1,0,1-c_0), & \hbox{if \ $a+c=1$, $b+d\neq1$}\\
                                              (0,a_0,x_3,1-a_0-x_3;0,c_0,y_3,1-c_0-y_3) \ or \\ (a_0,0,x_3,1-a_0-x_3;c_0,0,y_3,1-c_0-y_3), & \hbox{if \ $a+c\neq1$, $b+d=1$}\\
                                              (x_1,a_0-x_1,,x_3,1-a_0-x_3;y_1,c_0-y_1,y_3,1-c_0-y_3), & \hbox{if \ $a+c=1$, $b+d=1$}\\
\end{array}\right. \end{array}
\end{equation}
where   $$y_1=\frac{cc_0x_1}{aa_0+(c-a)x_1}, y_3=\frac{d(1-c_0)x_3}{b(1-a_0)+(d-b)x_3.}$$
\end{pro}

\subsection{Type of fixed points}\

Now we shall examine the type of the fixed points.
Before analyzing the fixed points we give the following useful lemma (\cite{ChX}).
\begin{lemma}\label{F(l)}
Let $F(\lambda)=\lambda^2+B\lambda+C,$ where $B$ and $C$ are two real constants. Suppose $\lambda_1$ and $\lambda_2$ are two roots of $F(\lambda)=0$. Then the following statements hold.
\begin{enumerate}
  \item[(i)] If $F(1)>0$ then
  \item[(i.1)] $|\lambda_1|<1$ and $|\lambda_2|<1$ if and only if $F(-1)>0$ and $C<1;$
  \item[(i.2)] $\lambda_1=-1$ and $\lambda_2\neq-1$ if and only if $F(-1)=0$ and $B\neq2;$
  \item[(i.3)] $|\lambda_1|<1$ and $|\lambda_2|>1$ if and only if $F(-1)<0;$
  \item[(i.4)] $|\lambda_1|>1$ and $|\lambda_2|>1$ if and only if $F(-1)>0$ and $C>1;$
  \item[(i.5)] $\lambda_1$ and $\lambda_2$ are a pair of conjugate complex roots and $|\lambda_1|=|\lambda_2|=1$ if only if $-2<B<2$ and $C=1;$
  \item[(i.6)] $\lambda_1=\lambda_2=-1$ if only if $F(-1)=0$ and $B=2.$
  \item[(ii)] If $F(1)=0,$ namely, 1 is one root of $F(\lambda)=0,$ then the other root $\lambda$ satisfies $|\lambda|=(<, >)1$  if and only if $|C|=(<, >)1.$
  \item[(iii)] If $F(1)<0$ then $F(\lambda)=0$ has one root lying in $(1;\infty).$ Moreover,
  \item[(iii.1)] the other root $\lambda$ satisfies $\lambda<(=)-1$ if and only if $F(-1)<(=)0;$
  \item[(iii.2)] the other root $\lambda$ satisfies $-1<\lambda<1$ if and only if $F(-1)>0.$
\end{enumerate}
\end{lemma}

\begin{pro}
For the fixed points of the operator (\ref{wxykso}), the followings hold true
$$(0,0) \ is\ \left\{\begin{array}{lll}
attractive, & \hbox{if \ $a+c<1$}\\
saddle, & \hbox{if \ $a+c>1$}\\
\end{array}\right. $$
$$(a_0,c_0) \ is\ \left\{\begin{array}{lll}
attractive, & \hbox{if \ $a+c>1$}\\
saddle, & \hbox{if \ $a+c<1$}\\
\end{array}\right. $$
$$\left(x, \frac{cc_0x}{aa_0+(c-a)x} \right)\  is\  non-hyperbolic,\  if\  a+c=1.$$
\end{pro}
\begin{proof}
To find the type of a fixed point $(x,y)$ of the operator (\ref{wxykso}) we write the Jacobi matrix:
$$J_{W_1}=\left(%
\begin{array}{cc}
  1-(1-a)c_0+(1-2a)y & aa_0+(1-2a)x \\
  cc_0+(1-2c)y & 1-(1-c)a_0+(1-2c)x \\
\end{array}%
\right).$$
The characteristic equation is
\begin{equation}
F(\lambda)=\lambda^2+B\lambda+C,
\end{equation}
where $$B=(1-a)c_0+(1-c)a_0-(1-2a)y-(1-2c)x-2,$$
$$C=(1-a-c)(a_0c_0-c_0x-a_0y)-(1-a)c_0-(1-c)a_0+(1-2a)y+(1-2c)x+1.$$

$\textbf{(1)}$ Let $x=0, y=0.$ Then $$F(1)=(1-a-c)a_0c_0,\ C=1-c_0(1-a(1-a_0))-a_0(1-c(1-c_0))<1,$$
$$F(-1)=3-(1-a)c_0-(1-c)a_0+(1-c_0)(1-a_0)+ac_0(1-a_0)+ca_0(1-c_0)>0.$$
If $a+c<1$, then $F(1)>0.$ According to item (i.1) of Lemma \ref{F(l)}, $|\lambda_{1,2}|<1.$ If $a+c>1$, then $F(1)<0.$ According to item (iii.2) of Lemma \ref{F(l)}, $\lambda_1>1,\ -1<\lambda<1.$

$\textbf{(2)} $ Let $x=a_0, y=c_0.$ Then
$$F(1)=-(1-a-c)a_0c_0,\ C=1-a_0c_0-ac_0(1-a_0)-a_0c(1-c_0)<1,$$
$$F(-1)=(1-c_0)(1-a_0)+ac_0(1-a_0)+ca_0(1-c_0)+2-ac_0-ca_0>0.$$
If $a+c>1$, then $F(1)>0.$ According to item (i.1) of Lemma \ref{F(l)}, $|\lambda_{1,2}|<1.$ If $a+c<1$, then $F(1)<0.$ According to item (iii.2) of Lemma \ref{F(l)}, $\lambda_1>1,\ -1<\lambda<1.$

$\textbf{(3)} $ Let $y=\frac{cc_0x}{aa_0+(c-a)x}.$ Then $F(1)=0.$ By item (ii) of Lemma \ref{F(l)}, at least one of the eigenvalues $\lambda_1$ and $\lambda_2$ is equal to one.
\end{proof}

\subsection{The set of limit points}\

Let $a+c\neq1.$

The following theorem  describes the trajectory of any point $(x^{(0)}, y^{(0)})$ in $S_{1}$.
\begin{thm}\label{wset} For the operator $W_{1}$ given by (\ref{wxykso}) (i.e. under condition (\ref{parshart})) and for any initial point $(x^{(0)}, y^{(0)})\in S_{1}\setminus Fix(W_{1})$ the following hold
$$\lim_{n\to \infty}x^{(n)}=\left\{\begin{array}{ll}
0, \ \ \ \ \ \mbox{if} \ \ a+c<1, \\[2mm]
a_0, \ \ \ \ \mbox{if} \ \ a+c>1,
\end{array}\right.$$
$$\lim_{n\to \infty}y^{(n)}=\left\{\begin{array}{ll}
0, \ \ \ \ \ \ \mbox{if} \ \ a+c<1, \\[2mm]
c_0, \ \ \ \ \ \mbox{if} \ \ a+c>1,
\end{array}\right.$$
where $(x^{(n)}, y^{(n)})=W_{1}^n(x^{(0)}, y^{(0)})$, with $W_{1}^n$ is $n$-th iteration of $W_{1}$.
\end{thm}
\begin{proof}
Let $a+c<1.$ Then there exists $k>1$ such that $c\cdot k=1-a$.
Denote
$z^{(n)}=x^{(n)}+y^{(n)}$ and $z^{(n)}_{0}=x^{(n)}+k\cdot y^{(n)}$, where $x^{(n)}, y^{(n)}$ defined by the following
\begin{equation}\label{recc}\begin{array}{ll}
x^{(n)}=x^{(n-1)}-(1-a)x^{(n-1)}(c_0-y^{(n-1)})+a(a_0-x^{(n-1)})y^{(n-1)},\\[3mm]
 y^{(n)}=y^{(n-1)}-(1-c)(a_0-x^{(n-1)})y^{(n-1)}+cx^{(n-1)}(c_0-y^{(n-1)}).
 \end{array}\end{equation}
By (\ref{recc}) we have
\begin{equation}\begin{array}{ll}
z^{(n)}-z^{(n-1)}=(a+c-1)\left(x^{(n-1)}(c_0-y^{(n-1)})+y^{(n-1)}(a_0-x^{(n-1)})\right)\leq0,\\[3mm]
z_0^{(n)}-z_0^{(n-1)}=(1-k)y^{(n-1)}(a_0-x^{(n-1)})\leq0.
\end{array}\end{equation}
Both sequences $\{z^{(n)}\}$ and $\{z^{(n)}_{0}\}$ are monotone and bounded,
i.e.,
$$0\leq...\leq z^{(n)}\leq z^{(n-1)}\leq...\leq z^{(0)},$$
$$ 0\leq...\leq z^{(n)}_{0}\leq z^{(n-1)}_{0}\leq...\leq z^{(0)}_{0}.$$
Thus $\{z^{(n)}\}$ and $\{z^{(n)}_{0}\}$ have limit point, denote the limits by $z^*$ and $z^*_{0}$ respectively.
Consequently, the following limits exist
$$y^*=\lim_{n\to \infty}y^{(n)}=\frac{1}{1-k}\lim_{n\to \infty}(z^{(n)}-z^{(n)}_{0})=\frac{1}{1-k}(z^*-z^*_{0}),$$
$$x^*=\lim_{n\to \infty}x^{(n)}=z^*-y^*.$$
 and by (\ref{recc}) we have
$$(1-a)x^*(c_0-y^*)=a(a_0-x^*)y^*,\ \ \ (1-c)(a_0-x^*)y^*=cx^*(c_0-y^*),$$
i.e., $x^*=0, y^*=0$ (see Figure 1).

For $a+c>1,$ the proof of the theorem is similar as above (see Figure 2).
\end{proof}
\begin{center}\label{fig1}
   \begin {tikzpicture} [scale=0.60]
   %Coordinates
   \draw[->, thick] (0,0) -- (0,8);
   \draw[->, thick] (0,0) -- (11,0);

   \draw[thick] (0,7) --  (10, 7);
   \draw[thick] (10,0) --  (10, 7);
   \draw[->, thick, dashed] (9.9,6.9) --  (0.1, 0.1);
   \draw[->, thick,dashed] (6,4.2) --  (5.9, 4.1);
   \draw[->, thick,dashed] (2.2,1.54) --  (2.1, 1.44);
    \draw[->, thick,dashed] (4.1,2.87) --  (4, 2.77);
    \draw[->, thick,dashed] (8.2,5.74) --  (8.1, 5.64);
  \draw[->, thick,dashed] (0.5,6.5).. controls (2,5.5) and (3,4) ..   (2, 2);
  \draw[->, thick,dashed] (2.5,6.5).. controls (4,5.5) and (5,5) ..   (4, 3.5);
  \draw[->, thick,dashed] (5.5,6.5).. controls (5.8,4.8)   ..   (5, 4);
\draw[->, thick,dashed] (9.5, 0.5).. controls (8,5) and (6,4) ..   (5.5, 3.5);
\draw[->, thick,dashed] (7.5, 0.5).. controls (6,2.5) ..   (3.5, 2.2);
\draw[->, thick,dashed] (5.5, 0.5).. controls (3.5,1.5) ..   (2.5, 1.2);

\filldraw[red] (0,0) circle (2pt);
\filldraw[red] (10,7) circle (2pt);
    %nodes x axes
   \node[left] at (10.5,-0.5){$a_0$};
   \node[above] at (11,0){$x$};
    %nodes y axes
   \node[left] at (0,7){$c_0$};
   \node[right] at (0,8){$y$};
    \node[left] at (1,-0.5){$(0,0)$};
   \node[above] at (10,7){$(a_0,c_0)$};
\end{tikzpicture}
\begin {tikzpicture} [scale=0.60]
   %Coordinates
   \draw[->, thick] (0,0) -- (0,8);
   \draw[->, thick] (0,0) -- (11,0);

   \draw[thick] (0,7) --  (10, 7);
   \draw[thick] (10,0) --  (10, 7);
   \draw[->, thick, dashed] (0.1, 0.1) -- (9.9,6.9);
   \draw[->, thick,dashed] (5.85, 4.1) -- (6,4.2);
   \draw[->, thick,dashed] (2.05, 1.44) -- (2.2,1.54);
    \draw[->, thick,dashed] (3.95, 2.77) -- (4.1,2.87);
    \draw[->, thick,dashed] (8.1, 5.64) --  (8.2,5.74);

\filldraw[red] (0,0) circle (2pt);
\filldraw[red] (10,7) circle (2pt);

  \draw[->, thick,dashed] (0.5,5.5).. controls (2.5,3.5) and (3,4) ..   (4, 4);
  \draw[->, thick,dashed] (2.5,6.5).. controls (4,5.5) and (5,5) ..   (6, 5);
  \draw[->, thick,dashed] (5.5,6.5).. controls (6.8,5.8)   ..   (8, 6);

\draw[->, thick,dashed] (9.5, 0.5).. controls (7,3)  ..   (8.5, 5.5);
\draw[->, thick,dashed] (6.5, 0.5).. controls (5.5,2.3) ..   (7, 4.2);
\draw[->, thick,dashed] (3.5, 0.5).. controls (3.5,1.5) ..   (4.5, 2.6);
    %nodes x axes
   \node[left] at (10.5,-0.5){$a_0$};
   \node[above] at (11,0){$x$};
    %nodes y axes
   \node[left] at (0,7){$c_0$};
   \node[right] at (0,8){$y$};
    \node[left] at (1,-0.5){$(0,0)$};
   \node[above] at (10,7){$(a_0,c_0)$};
\end{tikzpicture}

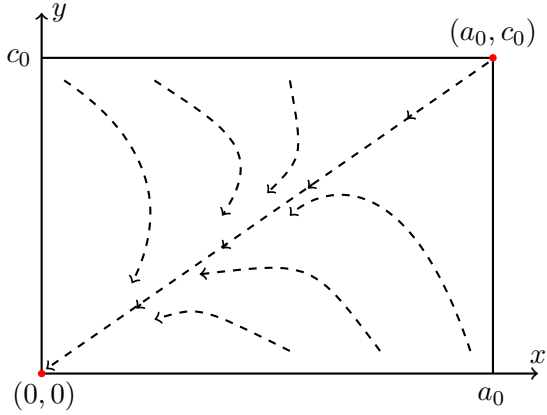
{Figure 1.} $a+c<1.$ \ \ \ \ \ \ \ \ \ \ \ \ \ \ \ \ \ \ \ \ \ \ \ \ \ \  {Figure 2.}\label{fig2} $a+c>1.$
\end{center}

Using Theorem \ref{wset}, the following theorem  describes the trajectory of any point $(u^{(0)}, v^{(0)})$ in $S_{2}=\{(u,v)\in\mathbb{R}^2: u\in[0,1-a_0], v\in[0,1-c_0]\}$.
\begin{thm}\label{wuvset} For the operator $W_{2}$ (i.e. under condition (\ref{parshart})) and for any initial point $(u^{(0)}, v^{(0)})\in S_{2}\setminus Fix(W_{2})$ the following hold
$$\lim_{n\to \infty}u^{(n)}=\left\{\begin{array}{ll}
0, \ \ \ \ \ \mbox{if} \ \ b+d<1, \\[2mm]
1-a_0, \ \ \ \ \mbox{if} \ \ b+d>1,
\end{array}\right.$$
$$\lim_{n\to \infty}v^{(n)}=\left\{\begin{array}{ll}
0, \ \ \ \ \ \ \mbox{if} \ \ b+d<1, \\[2mm]
1-c_0, \ \ \ \ \ \mbox{if} \ \ b+d>1,
\end{array}\right.$$
where $(u^{(n)}, v^{(n)})=W_2^n(u^{(0)}, v^{(0)})$, with $W_2^n$ is $n$-th iteration of $W_2$.
\end{thm}
The following theorem  describes the trajectory of any point $z$ in $I_{a_0c_0}$.
\begin{thm}\label{v2set} For the operator $V$ given by (\ref{exam2.kso}) (i.e. under condition (\ref{parshart})) and for any initial point $z=\left(x_1^{(0)},x_2^{(0)},x_3^{(0)},x_4^{(0)}; y_1^{(0)},y_2^{(0)},y_3^{(0)},y_4^{(0)}\right)\in I_{a_0c_0}\setminus Fix(V)$ the following hold
$$\lim_{n\to \infty}V^{(n)}\left(z\right)=\left\{\begin{array}{ll}
(0,a_0,0,1-a_0;0,c_0,0,1-c_0) \ \ \ \ \ \mbox{if} \ \ a+c<1,\  b+d<1, \\[2mm]
(0,a_0,1-a_0,0;0,c_0,1-c_0,0) \ \ \ \ \ \mbox{if} \ \ a+c<1,\  b+d>1, \\[2mm]
(a_0,0,0,1-a_0;c_0,0,0,1-c_0) \ \ \ \ \ \mbox{if} \ \ a+c>1,\  b+d<1, \\[2mm]
(a_0,0,1-a_0,0;c_0,0,1-c_0,0) \ \ \ \ \ \mbox{if} \ \ a+c>1,\  b+d>1. \\
\end{array}\right.$$
\end{thm}
\begin{proof}
Easily deduced from Theorem \ref{wset} and Theorem \ref{wuvset}.
\end{proof}

Let $a+c=1.$ Then, from (\ref{wxykso}) we get $x'+y'=x+y.$
Without loss of generality we assume $x+y=1$ in the operator $W_1$. Then we denoted by $T$, the following function
\begin{equation}\label{T}
T: x'=(2a-1)x^2+\left((1-a)(2-c_0)-aa_0\right)x+aa_0.
\end{equation}
Denote $$S^*=[0,1].$$
The following lemma is useful
\begin{lemma}
If conditions (\ref{parshart}) are satisfied then the function $T$ (defined by (\ref{T})) maps $S^*$ to itself.
\end{lemma}
\begin{proof} We want to show that if $x\in S^*$ then $x'=T(x)\in S^*$. $$T(0)=aa_0\in S^*,\ T(1)=1-c_0(1-a)\in S^*.$$ If $2a-1\geq0,$ then $T(x)=(2a-1)x^2+(1-a)(2-c_0)x+aa_0(1-x)\geq0.$
If $2a-1<0$, then the equation $T(x)-1=0$ has no real roots in the interval $(0,1)$. So, $T(x)\leq1$.
\end{proof}

Let us find fixed points of the function $T$. By solving the equation $T(x)=x,$ we get $t_1$,$t_2$ when $a\neq\frac{1}{2}$, and $t_3$ when $a=\frac{1}{2}$.
\begin{equation}\label{t1t2t3}
\begin{array}{lll}
t_1=\frac{1}{2(2a-1)}\left(1+aa_0-(1-a)(2-c_0)-\sqrt{D}\right),\\
t_2=\frac{1}{2(2a-1)}\left(1+aa_0-(1-a)(2-c_0)+\sqrt{D}\right),\\
t_3=\frac{a_0}{a_0+c_0},
\end{array}
\end{equation}
where $D=(1+aa_0-(1-a)(2-c_0))^2-4aa_0(2a-1).$

It is easy to check if $a>\frac{1}{2}$, then $0<t_1<1<t_2$, if $a<\frac{1}{2}$, then $t_2<0<t_1<1.$

Summarizing we formulate the following
\begin{pro}\label{T1} The set $Fix(T)$ has the following form
\begin{equation}\label{fixT}
\begin{array}{ll}
    Fix(T)=\left\{
                                            \begin{array}{lll}
                                              t_1, & \hbox{if \ $a\neq\frac{1}{2}$,}\\[1mm]
                                              t_3, & \hbox{if \ $a=\frac{1}{2}$,}\\
                                             \end{array}\right. \end{array}
\end{equation}
\end{pro}
Suppose $t^*$ is a fixed point for $T$. For one dimensional dynamical systems it is known that $t^*$ is an attracting fixed point if $|T'(t^*)|<1$. The point $t^*$ is a repelling fixed point if $|T'(t^*)|>1$. Finally, if $|T'(t^*)|=1$, the fixed point is saddle \cite{BZ}.

Let us calculate the derivative of $T'(x)$ at the fixed point $t^*$.
$$T'(t^*)=2(2a-1)t^*+(1-a)(2-c_0)-aa_0.$$
Putting the points $t_1$, $t_3$ instead of $t^*$ in $T'(t^*)$, we get the following
$$T'(t_1)=1-\sqrt{D},\ T'(t_3)=\frac{1}{2}(1-a_0-c_0).$$
Thus for type of $t^*$ the following proposition holds.
\begin{pro}\label{T2} The type of the fixed point $t^*$ for (\ref{T}) is attracting.
\end{pro}

A point $x$ in $T$ is called periodic point of $T$ if there exists $p$ so that $T^p(x)=x$. The smallest positive integer $p$ satisfying $T^p(x)=x$ is called the prime period or least period of   the point $x$.
\begin{pro}\label{T3} For $p\geq2$ the function (\ref{T}) does not have any $p$-periodic point in the set $S^*$.
\end{pro}
\begin{proof}
Let us first describe periodic points with $p=2$ on $S^*,$ i.e.,to solution of the equation
\begin{equation}\label{per.2}
T(T(x))=x.\end{equation}
Note that the fixed points of $T$ are solutions to (\ref{per.2}), to find other solution we consider the equation
$$\frac{T(T(x))-x}{T(x)-x}=0,$$
simple calculations show that the last equation is equivalent to the following
\begin{equation}\label{per.22}(2a-1)^2x^{2}+(2a-1)(1-aa_0+(1-a)(2-c_0))x+(1-a)(2-c_0)+2a^2a_0+1=0.\end{equation}
If $a=\frac{1}{2}$, then equation (\ref{per.22}) has no solution in $S^*$ .

Consider
$$l(x)=(2a-1)^2x^{2}+(2a-1)(1-aa_0+(1-a)(2-c_0))x+(1-a)(2-c_0)+2a^2a_0+1,$$
$$l'(x)=2(2a-1)^2+(2a-1)(1-aa_0+(1-a)(2-c_0)),$$ $$l''(x)=2(2a-1)^2.$$
By $N(c)$ we denote the number of variations in sign in the ordered sequence of numberes
$$l(c),\ l'(c),\ l''(c).$$
If $a>\frac{1}{2}$ then $N(0)=0$ and $N(1) = 0$. If $a<\frac{1}{2}$ then $N(0) = 2$ and $N(1) = 2$. According to the Budan-Foruier theorem, equation (\ref{per.22}) has no real root in the set $S^*$. Thus function (\ref{T}) does not have any 2-periodic point in $S^*$. Since $T$ is continuous on $S^*$ by Sharkovskii's theorem (\cite{R}) we have that $T^p(x)=x$ does
not have solution (except fixed) for all $p\geq2$.
\end{proof}
The following theorem describes the trajectory of any point $x_{0}$ in $S^*.$
\begin{thm}\label{wsetT} For the function $T$ given by (\ref{T}) (i.e. under condition (\ref{parshart})) and for any initial point $x_0\in S^*\setminus Fix(T)$ the following hold
$$\lim_{n\to \infty}x^{(n)}=\left\{\begin{array}{lll}
                                              t_1, & \hbox{if \ $a\neq\frac{1}{2}$,}\\[1mm]
                                              t_3, & \hbox{if \ $a=\frac{1}{2}$,}\\[1mm]
                                              %0, & \hbox{if \ $a>\frac{1}{2}$,\ $a_0=0, c_0\neq0$,\ or \ $\frac{1-c_0}{2-c_0}\leq a<\frac{1}{2}$,\ $a_0=0, c_0\neq0$,}\\ [1mm]
                                               % &  \hbox{or\ \ $a>\frac{1}{2}$,\ $a_0=0, c_0=0$,} \\
                                              %1, & \hbox{if \ $a<\frac{1}{2}$,\ $a_0\neq0, c_0=0$,\ or \ $\frac{1}{2}<a\leq\frac{1}{2-a_0}$,\ $a_0\neq0, c_0=0$,}\\ [1mm]
                                               % &  \hbox{or\ \ $a<\frac{1}{2}$,\ $a_0=0, c_0=0$,} \\
                                            %t_4, & \hbox{if \ $a<\frac{1-c_0}{2-c_0}$,\ $a_0=0, c_0\neq0$,}\\ [1mm]
                                            % t_5, & \hbox{if \ $a>\frac{1}{2-a_0}$,\ $a_0\neq0, c_0=0$.}\\
\end{array}\right.$$
where $t_1$ and $t_3$ are defined by (\ref{t1t2t3}) and $x^{(n)}=T^n(x_{0})$, with $T^n$ is $n$-th iteration of $T$.
\end{thm}
\begin{proof} The proof follows from Propositions \ref{T1}, \ref{T2}, and \ref{T3}.
\end{proof}
By using Theorem \ref{wsetT}, when  $a+c=1$ it is possible to describe the dynamics of the operator  $W_1$ at the set $S_1$ (see Figure 3).
\begin{center}
   \begin {tikzpicture} [scale=0.60]
   %Coordinates
   \draw[->, thick] (0,0) -- (0,8);
   \draw[->, thick] (0,0) -- (11,0);

   \draw[thick] (0,7) --  (10, 7);
   \draw[thick] (10,0) --  (10, 7);
   \draw[->, thick, red, dashed] (0.1, 0.1) .. controls (2.5,3.5) and (3,4) .. (9.9,6.9);
   \draw[thick,dashed] (0, 7) -- (10,0);
   \draw[thick,dashed] (0, 2.3) -- (3.3,0);
   \draw[thick,dashed] (0, 4.6) -- (6.6,0);
   \draw[thick,dashed] (3.3, 7) -- (10,2.3);
   \draw[thick,dashed] (6.6, 7) -- (10, 4.6);

   \draw[->, thick,dashed] (2, 5.6) -- (2.1,5.5);
   \draw[->, thick,dashed] (6.6, 2.38) -- (6.55,2.42);
    \draw[->, thick,dashed] (5, 3.5) -- (4.9,3.55);
    \draw[->, thick,dashed] (8.2, 1.26) --  (8.1,1.3);
\draw[->, thick,dashed] (0.6, 1.88) --  (0.7,1.78);
\draw[->, thick,dashed] (1, 3.90) --  (1.1,3.8);
\draw[->, thick,dashed] (3.3, 2.3) --  (3.2,2.35);
\draw[->, thick,dashed] (5, 1.12) --  (4.9,1.16);
\draw[->, thick,dashed] (5, 5.81) --  (5.1,5.71);
\draw[->, thick,dashed] (7, 4.4) --  (6.9,4.44);
\draw[->, thick,dashed] (8.3, 3.49) --  (8.2,3.55);
\draw[->, thick,dashed] (8.6, 5.59) --  (8.5,5.65);
\draw[->, thick,dashed] (7, 6.72) --  (7.1,6.65);
\draw[->, thick,dashed] (2, 0.91) --  (1.9,0.96);
\filldraw[red] (0,0) circle (2pt);
\filldraw[red] (10,7) circle (2pt);
\filldraw[red] (2.4,2.93) circle (2pt);
\filldraw[red] (1.15,1.5) circle (2pt);
\filldraw[red] (4.02,4.18) circle (2pt);
\filldraw[red] (5.95,5.15) circle (2pt);
\filldraw[red] (7.95,6.05) circle (2pt);

    %nodes x axes
   \node[left] at (10.5,-0.5){$a_0$};
   \node[above] at (11,0){$x$};
    %nodes y axes
   \node[left] at (0,7){$c_0$};
   \node[right] at (0,8){$y$};
   \node[left] at (0,0){$(0,0)$};
   \node[right] at (10,7){$(a_0,c_0)$};
\end{tikzpicture}

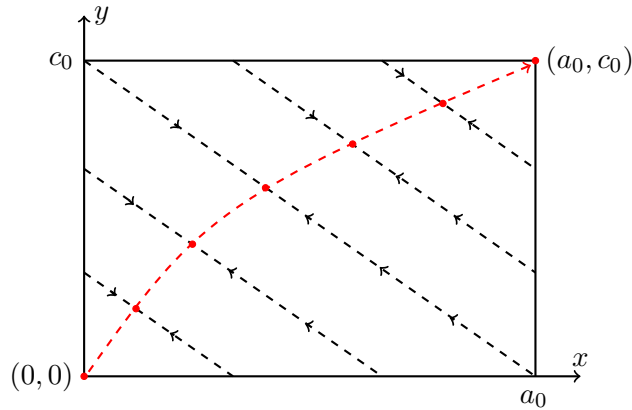
{Figure 3.}\label{fig.3} $a+c=1.$
\end{center}

\section{\bf{Biological interpretation of the results}}

In biology, a population genetics is the study of the distributions and changes of allele (type) frequency in a population.

The results formulated in previous sections have the following biological interpretations:

Let $z^{(0)}=(x^{(0)}, y^{(0))}\in S$ be an initial state, i.e., $x^{(0)}$(resp. $y^{(0)}$) is the probability distribution on the set $\{1, ..., n\}$ (resp. $\{1, ..., \nu\})$ of genotypes. Assume the trajectory $z^{(m)}$ of this point has a limit $z=((x_1, ..., x_n), (y_1, ..., y_\nu))$ this means that the future of the population is stable: female genotypes $i$ survives with probability $x_i$ and male genotype $j$ survives with probability $y_j$. A genotype will disappear if its probability is zero.

The case of two vertices ($n=\nu=2$):
\begin{itemize}
  \item [(a)] The population has a continuum set of equilibrium states (Proposition \ref{wfix}), it stays in a neighborhood of one of the equilibrium states of the population (stable fixed point).
  \item [(b)]  Let $a$ and $b$ be the probability of birth of type 1 female and  type 1 male from a type 1 male and a type 2 female, respectively.
\item [(b.1)] If the female and male of type 1 have probabilities $x^{(0)}$ and $y^{(0)}$, respectively, and the probability of female of type 2 is greater than $ay^{(0)}/(1-b)$, then  in the future evolution, both types of female and the  type 2 of male survive with the probability  $x^{(0)}+ay^{(0)}/(1-b)$, $1-x^{(0 )}-ay^{(0)}/(1-b)$ and 1, correspondingly. As a result, type 1 of the male disappears.
\item [(b.2)] If the female and male of type 1 have probabilities $x^{(0)}$ and $y^{(0)}$, respectively, and the probability of female of type 2 is less than $ay^{(0)}/(1-b)$, then  in the future evolution, both types of male and the  type 1 of female survive with the probability  $(1-b)(x^{(0)}-1)/a+y^{(0)}$, $1-(1-b)(x^{(0)}-1)/a+y^{(0)}$ and 1, correspondingly. As a result, type 2 of the female disappears.
\item [(b.3)]If the female and male of type 1 have probabilities $x^{(0)}$ and $y^{(0)}$, respectively, and the probability of female of type 2 equals  $ay^{(0)}/(1-b)$, then  in the future evolution,  type 2 of male and the  type 1 of female survive with the probability   1, correspondingly. As a result, type 2 of the female and type 1 of male disappear(interpretation of  Theorem \ref{v1set}).
\end{itemize}
The case of one edge and three vertices ($n=\nu=4$):
\begin{itemize}
  \item [(c)] Depending on the parameters the population my have two or infinitely many (continuum) equilibrium states (Proposition \ref{pro5}). The population stays in a neighbor-hood of one of the equilibrium states of the population (stable fixed point).
\item[(d)] Let $a$ and $c$ be the probability of birth of type 1 female and  type 1 male from either a type 1 male and a type 2 female or a type 2 male and a type 1 female, respectively. Let $b$ and $d$ be the probability of birth of type 3 female and  type 3 male from either a type 3 male and a type 4 female or a type 4 male and a type 3 female, respectively.
\item [(d.1)] If $a+c<1,\b+d<1$, in future evolution, the  types 2 and 4 of female and male survive with the probability $a_0,\ 1-a_0$, $c_0,\ 1-c_0$, respectively, and types 1 and 3 of the female and male disappear.
\item [(d.2)] If $a+c<1,\ b+d>1$, in future evolution, the  types 2 and 3 of female and male survive with the probability $a_0,\ 1-a_0$, $c_0,\ 1-c_0$, respectively, and types 1 and 4 of the female and male disappear.
\item [(d.3)] If $a+c>1,\ b+d<1$, in future evolution, the  types 1 and 4 of female and male survive with the probability $a_0,\ 1-a_0$, $c_0,\ 1-c_0$, respectively, and types 2 and 3 of the female and male disappear.
\item [(d.4)] If $a+c>1,\ b+d>1$,  in future evolution, the  types 1 and 3 of female and male survive with the probability $a_0,\ 1-a_0$, $c_0,\ 1-c_0$, respectively, and types 2 and 4 of the female and male disappear(interpretation of  Theorem \ref{v2set}).
\end{itemize}

\subsection{Acknowledgement.}  I am grateful to Professor U.A.Rozikov for the comments and suggestions.


\begin{thebibliography}{99}
\bibitem{B} N. Bernshtein, \textit{Solution of one mathematical problem related to the theory of inheritance}, (1924) Uch. Zap. Nauchno-Issled. Kaf. Ukr. Otd. Mat., 1, 83-115

\bibitem{P} C. Preston, \textit{Gibbs States on Countable Sets}, (1974) Cambridge University Press, London-New York.

\bibitem{L} Yu. I. Lyubich, \textit{Mathematical structures in population genetics}, (1992) Biomathematics, 22, Springer-Verlag, Berlin.

\bibitem{RNG1} R. N. Ganikhodzhaev, \textit{Quadratic stochastic operators, Lyapunov functions, and tournaments}, (1992) Mat. Sb. 183 (8), 119-140 [Russian Acad. Sci. Sb.Math. 76 (2), 489-506 (1993)].

\bibitem{RNG2} R. N. Ganikhodzhaev, \textit{Map of fixed points and Lyapunov functions for a class of discrete dynamical systems}, (1994) Mat. Zametki 56 (5), 40-49 [Math. Notes 56 (5), 1125-1131 (1994)].

\bibitem{D} R. L. Devaney,  \textit{An Introduction to Chaotic Dynamical System}, (2003), Westview Press.

\bibitem{GR} N. N. Ganikhodzhaev, U. A. Rozikov, \textit{On quadratic stochastic operators generated by Gibbs distributions}, (2006), Regular chaotic dynamics, 11(4), 467-473.

\bibitem{RNGandE} R. N. Ganikhodzhaev and D. B. E'shmamatova, \textit{Quadratic automorphisms of a simplex and the asymptotic behavior of their trajectories}, (2006), Vladikavkaz.Mat. Zh. 8 (2), 12-28.

\bibitem{RJ} U. A. Rozikov, U. U. Zhamilov, \textit{Volterra quadratic stochastic operators of a two-sex population}, (2011), Ukrainian Math. J. 63, no. 7, 1136-1153.

\bibitem{LR} M. Ladra, U.A. Rozikov, \textit{Evolution algebra of a bisexual population}, (2013), J. Algebra 378 153-172.

\bibitem{DOR} A. Dzhumadil'daev, B. A. Omirov, U. A. Rozikov, \textit{Constrained evolution algebras and dynamical systems of a bisexual population}, (2016), Linear Algebra and its Applications
496, 351-380

\bibitem{CJR} O. Castanos, U.U. Jamilov, U.A. Rozikov, \textit{On Volterra quadratic stochastic operators of a two-sex population on $S^1\times S^1$}, doi.org/10.48550/arXiv.1808.01812.

\bibitem{ChX} Cheng Wang, Xianyi Li, \textit{Stability and Neimark-Sacker bifurcation of a semi-discrete population model}, (2014), Journal of Applied Analysis and Computation, vol.4, no.4, 419-435

\bibitem{GJ} N. N. Ganikhodjaev, U. U. Jamilov \textit{Contracting Quadratic Operators of Bisexual population},(2015), Appl. Math. Inf. Sci. 9, No. 5, 2645-2650

\bibitem{R} U. A. Rozikov, \textit{An Introduction to Mathematical Billiards}, (2019), World Scientific Publishing, Singapore, 224 pp.

\bibitem{Rpd}  U.A. Rozikov, (2020), Population dynamics: algebraic and probabilistic approach, \textit{World Sci. Publ., Sing.}, 460 pp.

\bibitem{BR}  Z.S. Boxonov,  U.A.Rozikov, \textit{Dynamical system of a mosquito population with distinct birth-death rates}, (2021) Journal of Appleid Nonlinear Dynamics, Vol.10, No.4, p.807-816.

\bibitem{EJX} F.F. Eshmatov, U.U. Jamilov, Kh.O. Khudoyberdiev, \textit{ Discrete Time Dynamics of a SIRD Reinfection Model}, (2023), International Journal of Biomathematics. 16 (5), 2250104.

\bibitem{BZ} Z.S. Boxonov,\textit{A discrete-time dynamical system of mosquito population}, JDEA., (2023), Vol. 29, No. 1, p.67-83. doi.org/10.1080/10236198.2022.2160245

\end{thebibliography}
\end{document}